\documentclass[10pt]{amsart}
\usepackage[T1]{fontenc}
\usepackage[latin5]{inputenc}
\usepackage{float}
\usepackage{amssymb}
\usepackage{amsfonts}
\usepackage[centertags]{amsmath}
\usepackage{amsthm}
\usepackage{newlfont}
 \makeatletter
\theoremstyle{plain}
\newtheorem{thm}{Theorem}[section]
\numberwithin{equation}{section}
\numberwithin{figure}{section}  
\theoremstyle{plain}

\theoremstyle{plain}

\theoremstyle{plain}
\newtheorem{cor}[thm]{Corollary} 
\theoremstyle{plain}

\theoremstyle{plain}
\newtheorem{lem}[thm]{Lemma} 
\theoremstyle{plain}

\begin{document}
\title{Integrated and Differentiated Sequence Spaces}
\author{Murat Kiri\c{s}ci}
\address{Department of Mathematical Education, Hasan Ali Y\"{u}cel Education Faculty, Istanbul University, Vefa, 34470, Fatih, Istanbul, Turkey}
\email{mkirisci@hotmail.com, murat.kirisci@istanbul.edu.tr}
\vspace{0.5cm}
\thanks{This work was supported by Scientific Projects Coordination Unit of Istanbul University. Project number 34465.}
\vspace{0.5cm}

\begin{abstract}
In this paper, we investigate integrated and differentiated sequence spaces which emerge from the concept of the space $bv$ of sequences of bounded variation.The integrated and differentiated sequence spaces which was initiated by Goes and Goes \cite{Goes}. The main propose of the present paper,
we study of matrix domains and some properties of the integrated and differentiated sequence spaces. In section 3, we compute
the alpha-, beta- and gamma duals of these spaces. Afterward, we characterize the matrix classes of these spaces with well-known sequence spaces.
\end{abstract}

\maketitle

\section{Introduction}
The set of all sequences denotes with $\omega := \mathbb{C}^{\mathbb{N}}:=\{x=(x_{k}): x: \mathbb{N}\rightarrow \mathbb{C}, k\rightarrow x_{k}:=x(k)\}$ where $\mathbb{C}$ denotes the complex field and $\mathbb{N}=\{0,1,2,\ldots\}$. Each linear subspace of $\omega$ (with the induced addition and scalar multiplication) is called a \emph{sequence space}. The following subsets of $\omega$ are obviously sequence spaces:$\ell_{\infty}=\{x=(x_{k})\in \omega: \sup_{k}|x_{k}|<\infty\}$, $c=\{x=(x_{k})\in \omega: \lim_{k}x_{k} ~\textrm{ exists}~\}$, $c_{0}=\{x=(x_{k})\in \omega: \lim_{k}x_{k}=0\}$, $bs=\{x=(x_{k})\in \omega: \sup_{n}|\sum_{k=1}^{n}x_{k}|<\infty\}$, $cs=\{x=(x_{k})\in \omega: (\sum_{k=1}^{n}x_{k})\in c\}$ and $\ell_{p}=\{x=(x_{k})\in \omega: \sum_{k}|x_{k}|^{p}<\infty, \quad 1\leq p < \infty\}$. These sequence spaces are Banach space with the following norms; $\|x\|_{\ell_{\infty}}=\sup_{k}|x_{k}|$, $\|x\|_{bs}=\|x\|_{cs}=\sup_{n}|\sum_{k=1}^{n}x_{k}|$ and $\|x\|_{\ell_{p}}=\left(\sum_{k}|x_{k}|^{p}\right)^{1/p}$ as usual, respectively. And also the concept of integrated and differentiated sequence spaces was employed as $\int X=\left\{x=(x_{k})\in \omega: (kx_{k}) \in X \right\}$ and $d(X)=\left\{x=(x_{k})\in \omega: (k^{-1}x_{k}) \in X \right\}$ in \cite{Goes}.

A sequence, whose $k-th$ term is $x_{k}$, is denoted by $x$ or $(x_{k})$. \emph{A coordinate space} (or \emph{$K-$space}) is a vector space of numerical sequences, where addition and scalar multiplication are defined pointwise. That is, a sequence space $X$ with a linear topology is called a $K$-space provided each of the maps $p_{i}:X\rightarrow \mathbb{C}$ defined by $p_{i}(x)=x_{i}$ is continuous for all $i\in \mathbb{N}$. A $BK-$space is a $K-$space, which is also a Banach space with continuous coordinate functionals $f_{k}(x)=x_{k}$, $(k=1,2,...)$. A $K-$space $K$ is called an \emph{$FK-$space} provided $\lambda$ is a complete linear metric space. An \emph{$FK-$space} whose topology is normable is called a \emph{$BK-$ space}. If a normed sequence space $X$ contains a sequence $(b_{n})$ with the property that for every $x\in X$ there is unique sequence of scalars $(\alpha_{n})$ such that
\begin{eqnarray*}
\lim_{n\rightarrow\infty}\|x-(\alpha_{0}b_{0}+\alpha_{1}b_{1}+...+\alpha_{n}b_{n})\|=0
\end{eqnarray*}
then $(b_{n})$ is called \emph{Schauder basis} (or briefly basis) for $X$. The series $\sum\alpha_{k}b_{k}$ which has the sum $x$ is then called the expansion of $x$ with respect to $(b_{n})$, and written as $x=\sum\alpha_{k}b_{k}$. An \emph{$FK-$space} $X$ is said to have $AK$ property, if $\phi \subset X$ and $\{e^{k}\}$ is a basis for $X$, where $e^{k}$ is a sequence whose only non-zero term is a $1$ in $k^{th}$ place for each $k\in \mathbb{N}$ and $\phi=span\{e^{k}\}$, the set of all finitely non-zero sequences.

Let $X$ and $Y$ be two sequence spaces, and $A=(a_{nk})$ be an infinite matrix of complex numbers $a_{nk}$, where $k,n\in\mathbb{N}$. Then, we say that $A$ defines a matrix mapping from $X$ into $Y$, and we denote it by writing $A :X \rightarrow Y$ if for every sequence $x=(x_{k})\in X$. The sequence $Ax=\{(Ax)_{n}\}$, the $A$-transform of $x$, is in $Y$; where
\begin{eqnarray}\label{equ1}
(Ax)_{n}=\sum_{k}a_{nk}x_{k}~\textrm{ for each }~n\in\mathbb{N}.
\end{eqnarray}
For simplicity in notation, here and in what follows, the summation without limits runs from $1$ to $\infty$. By $(X:Y)$, we denote the class of all matrices $A$ such that $A:X \rightarrow Y $. Thus, $A\in(X:Y)$ if and only if the series on the right side of (\ref{equ1}) converges for each $n\in\mathbb{N}$ and each $x\in X$ and we have $Ax =\{(Ax)_{n}\}_{n \in\mathbb{N}}\in Y$ for all $x\in X$. A sequence $x$ is said to be $A$-summable to $l$ if $Ax$ converges to $l$ which is called the $A$-limit of $x$.\\

Let $X$ is a sequence space and $A$ is an infinite matrix. The sequence space
\begin{eqnarray}\label{eq0}
X_{A}=\{x=(x_{k})\in\omega:Ax\in X\}
\end{eqnarray}
is called the matrix domain of $X$ which is a sequence space(for several examples of matrix domains, see \cite{Basarkitap} p. 49-176). By $\mathcal{F}$, we will denote the collection of all finite subsets on $\mathbb{N}$. In \cite{AB2}, Ba\c{s}ar and Altay have defined the sequence space $bv_{p}$ which consists of all sequences such that $\Delta$-transforms of them are in $\ell_{p}$ where $\Delta$ denotes the matrix $\Delta=(\delta_{nk})$
\begin{eqnarray*}
\delta_{nk}= \left\{ \begin{array}{ccl}
(-1)^{n-k}&, & \quad (n-1\leq k \leq n)\\
0&, & \quad (0\leq k < n-1 ~\textrm{ or }~ k>n)
\end{array}\right.
\end{eqnarray*}
for all $k,n\in \mathbb{N}$. And also we define the matrices $\Gamma=(\gamma_{nk})$ and $\Sigma=(\sigma_{nk})$ by
\begin{eqnarray}\label{matr1}
\gamma_{nk}= \left\{ \begin{array}{ccl}
k&, & \quad (n=k)\\
-k&, & \quad (n-1=k)\\
0&, & \quad (other)
\end{array} \right.
\end{eqnarray}

\begin{eqnarray}\label{matr2}
\sigma_{nk}= \left\{ \begin{array}{ccl}
\frac{1}{k}&, & \quad (n=k)\\
-\frac{1}{k}&, & \quad (n-1=k)\\
0&, & \quad (other)
\end{array} \right.
\end{eqnarray}

The integrated and differentiated sequence spaces which was initiated by Goes and Goes \cite{Goes}. In the present paper,
we study of matrix domains and some properties of the integrated and differentiated sequence spaces. In section 3, we compute
the alpha-, beta- and gamma duals of these spaces. Afterward, we characterize matrix classes of these spaces with well-known
sequence spaces.

\section{The Sequence Spaces $\int bv$ and $d(bv)$}

The integrated spaces defined by
\begin{eqnarray*}
\int \ell_{1}&=&\left\{x=(x_{k})\in \omega: \sum_{k}|k.x_{k}|<\infty\right\}\\
\int bv&=&\left\{x=(x_{k})\in \omega: \sum_{k}|k.x_{k}-(k-1).x_{k-1}|<\infty\right\}
\end{eqnarray*}
and the differentiated spaces defined by
\begin{eqnarray*}
d(\ell_{1})&=&\{x=(x_{k})\in \omega: \sum_{k}|k^{-1}.x_{k}|<\infty\}\\
d(bv)&=&\{x=(x_{k})\in \omega: \sum_{k}|k^{-1}.x_{k}-(k-1)^{-1}.x_{k-1}|<\infty\}.
\end{eqnarray*}

Consider the notation (\ref{eq0}) and the matrices (\ref{matr1}), (\ref{matr2}). From here, we can re-define the spaces $\int bv$ and $d(bv)$ by
\begin{eqnarray}\label{domain}
\left(\int \ell_{1}\right)_{\Delta}=\int bv ~\textrm{ or }~ \left(\ell_{1}\right)_{\Gamma}=\int bv
\end{eqnarray}
and
\begin{eqnarray}\label{domain1}
\left[d(\ell_{1})\right]_{\Delta}=d(bv) ~\textrm{ or }~ \left(\ell_{1}\right)_{\Sigma}=d(bv).
\end{eqnarray}

Let $x=(x_{k})\in \int bv$ and $\Delta x_{k}=x_{k}-x_{k-1}$. The $\Gamma-$transform of a sequence $x=(x_{k})$ is defined by
\begin{eqnarray}\label{transform1}
y_{k}=(\Gamma x)_{k}=\left\{ \begin{array}{ccl}
x_{1}&, & \quad k=1\\
\Delta(kx_{k})&, & \quad k\geq 2
\end{array} \right.
\end{eqnarray}
where $\Gamma$ is defined by (\ref{matr1}).
Let $x=(x_{k})\in d(bv)$ and $\Delta x_{k}=x_{k}-x_{k-1}$. The $\Sigma-$transform of a sequence $x=(x_{k})$ is defined by
\begin{eqnarray}\label{transform2}
y_{k}=(\Sigma x)_{k}=\left\{ \begin{array}{ccl}
x_{2}/2&, & \quad k=2\\
\Delta(k^{-1}x_{k})&, & \quad k\geq 3
\end{array} \right.
\end{eqnarray}
where $\Sigma$ is defined by (\ref{matr2}).

\begin{thm}
The spaces $\int\ell_{1}$ and $d(\ell_{1})$ are $BK-$spaces with the norms $\|x\|_{\int\ell_{1}}=\sum_{k}|kx_{k}|$ and
$\|x\|_{d(\ell_{1})}=\sum_{k}|k^{-1}x_{k}|$, respectively.
\end{thm}
\begin{proof}
Let $x=(x_{k})\in \int\ell_{1}$. We define $f_{k}(x)=x_{k}$ for all $k\in \mathbb{N}$. Then, we have
\begin{eqnarray*}
\|x\|_{\int\ell_{1}}=1.|x_{1}|+2.|x_{2}|+3.|x_{3}|+\cdots+k.|x_{k}|+\cdots
\end{eqnarray*}
Hence $k.|x_{k}|\leq \|x\|_{\int\ell_{1}} \Rightarrow |x_{k}|\leq K.\|x\|_{\int\ell_{1}} \Rightarrow |f_{k}(x)|\leq K.\|x\|_{\int\ell_{1}}$.
Then, $f_{k}$ is continuous linear functional for each $k$. Thus $\int\ell_{1}$ is a $BK-$space.\\

In a similar way, we can prove that the space $d(\ell_{1})$ is a $BK-$spaces.
\end{proof}

\begin{lem}\cite{Goes}
The space $\int bv$ is a $BK-$space with the norm $\|x\|_{\int bv}=\sum_{k}|\Delta(kx_{k})|$.
\end{lem}

\begin{thm}
The space $d(bv)$ is a $BK-$space with the norm $\|x\|_{d(bv)}=\sum_{k}|\Delta(k^{-1}x_{k})|$.
\end{thm}
\begin{proof}
Since $d(bv)=\left[d(\ell_{1})\right]_{\Delta}$ holds, $d(\ell_{1})$ is a $BK-$space with the norm $\|x\|_{d(\ell_{1})}$ and
the matrix $\Delta$ is a triangle matrix, then Theorem 4.3.2 of Wilansky\cite{Wil84} gives the fact that the space $d(bv)$ is
a $BK-$space.
\end{proof}

\begin{thm}
\begin{itemize}
\item [(i).] The spaces $\int\ell_{1}$ and $d(\ell_{1})$ have $AK-$property.
\item [(ii).] The spaces $\int bv$ and $d(bv)$ have $AK-$property.
\end{itemize}
\end{thm}

\begin{proof}
The fact that of the space $\int bv$ has $AK-$property was given by Goes and Goes\cite{Goes}. Then, we will only prove that the space $d(bv)$ has $AK-$property in (ii).\\

Let $x=(x_{k})\in d(bv)$ and $x^{[n]}=\{x_{1}, x_{2}, \cdots, x_{n},0,0,\cdots\}$.
Hence,
\begin{eqnarray*}
x-x^{[n]}=\{0,0,\cdots,0,x_{n+1}, x_{n+2},\cdots\} \Rightarrow \|x-x^{[n]}\|_{d(bv)}=\|0,0,\cdots,0,x_{n+1}, x_{n+2},\cdots\|
\end{eqnarray*}
and since $x\in d(bv)$,
\begin{eqnarray*}
\|x-x^{[n]}\|_{d(bv)}=\sum_{k\geq n+1}|\Delta(k^{-1}x_{k})|\rightarrow 0 ~\textrm{ as}~\ n\rightarrow \infty\\
\Rightarrow \lim_{n\rightarrow \infty}\|x-x^{[n]}\|_{d(bv)}=0 \Rightarrow x^{[n]}\rightarrow \infty ~\textrm{ as}~\ n\rightarrow \infty ~\textrm{ in}~\ d(bv).
\end{eqnarray*}
Then the space $d(bv)$ has $AK-$property.
\end{proof}

\begin{thm}\label{isomorph}
The spaces $\int bv$ and $d(bv)$ are norm isomorphic to $\ell_{1}$.
\end{thm}

\begin{proof}
We must show that a linear bijection between the spaces $\int bv$ and $\ell_{1}$ exists. Consider the transformation $T$ defined, with the notation (\ref{transform1}), from $\int bv$ to $\ell_{1}$ by $x \mapsto y=Tx$. The linearity of $T$ is clear. Also, it is trivial that $x=\theta$ whenever $Tx=\theta$ and therefore, $T$ is injective.\\

Let $y\in \ell_{1}$ and define the sequence $x=(x_{k})$ by $x_{k}=\frac{1}{k}.\sum_{j=1}^{k}y_{j}$. Then
\begin{eqnarray*}
\|x\|_{\int bv}&=&\sum_{k}|\Delta(kx_{k})|=\sum_{k}\left|k.\frac{1}{k}\sum_{j=1}^{k}y_{j}-(k-1).\frac{1}{k-1}\sum_{j=1}^{k-1}y_{j}\right|
=\sum_{k}|y_{k}|=\|y\|_{\ell_{1}}<\infty.
\end{eqnarray*}
Then, we have that $x\in \int bv$. So, $T$ is surjective and norm preserving. Hence $T$ is a linear bijection. It shown us that the space $\int bv$ is norm isomorphic to $\ell_{1}$.\\

As similar, using the notation (\ref{transform2}), we can define the transformation $S$ from $d(bv)$ and $\ell_{1}$ by $x \mapsto y=Sx$. And also, if we choose the sequence $x=(x_{k})$ by $x_{k}=k.\sum_{j=2}^{k}y_{j}$ while $y\in \ell_{1}$, then we obtain the space $d(bv)$ is norm isomorphic to $\ell_{1}$
with the norm $\|x\|_{d(bv)}$.
\end{proof}

\begin{thm}
The spaces $\int bv$ and $d(bv)$ have monotone norm.
\end{thm}
\begin{proof}
Let $x=(x_{k})\in \int bv$. We define the norms $\|x\|_{\int bv}=\sum_{k}|\Delta(kx_{k})|$ and $\|x^{[n]}\|_{\int bv}=\sum_{k=1}^{n}|\Delta(kx_{k})|$, for all $x\in \int bv$. For $n<m$,
\begin{eqnarray*}
\|x^{[n]}\|=\sum_{k=1}^{n}|\Delta(kx_{k})|\leq \sum_{k=1}^{m}|\Delta(kx_{k})=\|x^{[m]}\|,
\end{eqnarray*}
that is,
\begin{eqnarray}\label{mono}
\|x^{[m]}\|\geq \|x^{[n]}\|.
\end{eqnarray}
The sequence $\|x^{[n]}\|$ is monotonically increasing sequence and bounded above.
\begin{eqnarray}\label{mono1}
\sup\|x^{[n]}\|= \sup\left(\sum_{k=1}^{n}|\Delta(kx_{k})|\right)=\left(\sum_{k=1}^{n}|\Delta(kx_{k})|\right)=\|x\|.
\end{eqnarray}
From (\ref{mono}) and (\ref{mono1}), it follows that the space $\int bv$ has the monotone norm.\\

In similar way, we can obtain to the space $d(bv)$ has the monotone norm.
\end{proof}

Because of the isomorphisms $T$ and $S$, defined in the proof of Theorem \ref{isomorph}, are onto the inverse image of the basis $\{e^{(k)}\}_{k\in \mathbb{N}}$ of the space $\ell_{1}$ is the basis of the spaces $\int bv$ and $d(bv)$. Therefore, we have the following:

\begin{thm}
\begin{itemize}
\item [(i).] Define a sequence $t^{(k)}=\{t_{n}^{(k)}\}_{n\in\mathbb{N}}$ of elements of the space $\int bv$ for every fixed $k\in \mathbb{N}$ by
\begin{eqnarray*}
t_{n}^{(k)}= \left\{ \begin{array}{ccl}
1/k&, & \quad (n\geq k)\\
0&, & \quad (n<k)
\end{array} \right.
\end{eqnarray*}
Therefore, the sequence $\{t^{(k)}\}_{k\in\mathbb{N}}$ is a basis for the space $\int bv$ and if we choose $E_{k}=(Ax)_{k}$ for all $k\in\mathbb{N}$, where the matrix $A$ defined by (\ref{matr1}), then any $x\in \int bv$ has a unique representation of the form
\begin{eqnarray*}
x=\sum_{k}E_{k}b^{(k)}.
\end{eqnarray*}
\item [(ii).] Define a sequence $s^{(k)}=\{s_{n}^{(k)}\}_{n\in\mathbb{N}}$ of elements of the space $d(bv)$ for every fixed $k\in \mathbb{N}$ by
\begin{eqnarray*}
s_{n}^{(k)}= \left\{ \begin{array}{ccl}
k&, & \quad (n\geq k)\\
0&, & \quad (n<k)
\end{array} \right.
\end{eqnarray*}
Therefore, the sequence $\{s^{(k)}\}_{k\in\mathbb{N}}$ is a basis for the space $d(bv)$ and if we choose $F_{k}=(Bx)_{k}$ for all $k\in\mathbb{N}$, where the matrix $B$ defined by (\ref{matr2}), then any $x\in d(bv)$ has a unique representation of the form
\begin{eqnarray*}
x=\sum_{k}F_{k}b^{(k)}.
\end{eqnarray*}
\end{itemize}
\end{thm}

The result follows from fact that if a space has a Schauder basis, then it is separable. Hence, we can give following corollary:

\begin{cor}
The spaces $\int bv$ and $d(bv)$ are separable.
\end{cor}

\section{The $\alpha-$, $\beta-$ and $\gamma-$ Duals of the spaces $\int bv$ and $d(bv)$}
In this section, we state and prove the theorems determining the $\alpha$-, $\beta$- and $\gamma$-duals of the sequence spaces $\int bv$ and $d(bv)$.

Let $x$ and $y$ be sequences, $X$ and $Y$ be subsets of $\omega$ and $A=(a_{nk})_{n,k=0}^{\infty}$ be an infinite matrix of complex numbers. We write $xy=(x_{k}y_{k})_{k=0}^{\infty}$, $x^{-1}*Y=\{a\in\omega: ax\in Y\}$ and $M(X,Y)=\bigcap_{x\in X}x^{-1}*Y=\{a\in\omega: ax\in Y ~\textrm{ for all }~ x\in X\}$ for the \emph{multiplier space} of $X$ and $Y$. In the special cases of $Y=\{\ell_{1}, cs, bs\}$, we write $x^{\alpha}=x^{-1}*\ell_{1}$, $x^{\beta}=x^{-1}*cs$, $x^{\gamma}=x^{-1}*bs$ and $X^{\alpha}=M(X,\ell_{1})$, $X^{\beta}=M(X,cs)$, $X^{\gamma}=M(X,bs)$ for the $\alpha-$dual, $\beta-$dual, $\gamma-$dual of $X$. By $A_{n}=(a_{nk})_{k=0}^{\infty}$ we denote the sequence in the $n-$th row of $A$, and we write $A_{n}(x)=\sum_{k=0}^{\infty}a_{nk}x_{k}$ $n=(0,1,...)$ and $A(x)=(A_{n}(x))_{n=0}^{\infty}$, provided $A_{n}\in x^{\beta}$ for all $n$.\\

\begin{lem}\label{duallem1}\cite[Theorem 2.1]{AB}
Let $\lambda, \mu$ be the BK-spaces and $B_{\mu}^{U}=(b_{nk})$ be defined via the sequence $\alpha=(\alpha_{k})\in\mu$ and triangle matrix $U=(u_{nk})$ by
\begin{eqnarray*}
b_{nk}=\sum_{j=k}^n\alpha_{j}u_{nj}v_{jk}
\end{eqnarray*}
for all $k,n\in\mathbb{N}$. Then, the inclusion $\mu\lambda_{U}\subset\lambda_{U}$ holds if and only if the matrix $B_{\mu}^{U}=UD_{\alpha}U^{-1}$ is in the classes $(\lambda:\lambda)$, where $D_{\alpha}$ is the diagonal matrix defined by $[D_{\alpha}]_{nn}=\alpha_{n}$ for all $n\in\mathbb{N}$.
\end{lem}

\begin{lem}\label{duallem2}\cite[Theorem 3.1]{AB}
$B_{\mu}^{U}=(b_{nk})$ be defined via a sequence $a=(a_{k})\in\omega$ and inverse of the triangle matrix $U=(u_{nk})$ by
\begin{eqnarray*}
b_{nk}=\sum_{j=k}^na_{j}v_{jk}
\end{eqnarray*}
for all $k,n\in\mathbb{N}$. Then,
\begin{eqnarray*}
\lambda_{U}^{\beta}=\{a=(a_{k})\in\omega: B^{U}\in(\lambda:c)\}.
\end{eqnarray*}
and
\begin{eqnarray*}
\lambda_{U}^{\gamma}=\{a=(a_{k})\in\omega: B^{U}\in(\lambda:\ell_{\infty})\}.
\end{eqnarray*}
\end{lem}

\begin{lem}\label{duallem3}
Let $A=(a_{nk})$ be an infinite matrix.  Then, the following statements hold:
\begin{itemize}
\item[(i)] $A\in (\ell_{1}:\ell_{\infty})$ if and only if
\begin{eqnarray} \label{deq1}
\sup_{k,n \in \mathbb{N}}|a_{nk}| < \infty.
\end{eqnarray}

\item[(ii)] $A\in (\ell_{1}:c)$ if and only if (\ref{deq1}) holds, and there are $\alpha_{k},\in \mathbb{C}$ such that
\begin{eqnarray}\label{deq2}
\lim_{n \rightarrow \infty} a_{nk} = \alpha_{k}~ \textrm{ for each }~k\in\mathbb{N}.
\end{eqnarray}

\item[(iii)] $A\in (\ell_{1}:\ell_{1})$ if and only if
\begin{eqnarray}\label{deq3}
\sup_{k \in \mathbb{N}}\sum_{n}|a_{nk}| < \infty.
\end{eqnarray}

\end{itemize}
\end{lem}

\begin{thm}\label{alphathm}
$\left[\int bv\right]^{\alpha}=d(\ell_{1})$
\end{thm}

\begin{proof}
We take the matrix $\Gamma$ as defined by (\ref{matr1}) and $\Gamma_{n}$ denotes the sequences in the $n$th rows of the matrices $\Gamma$. We define the matrix $C$ whose rows are the product of the rows of the matrix $\Gamma^{-1}$ and the sequence $a=(a_{n})$, i.e., $C_{n}=(\Gamma^{-1})_{n}a$. From the relation (\ref{transform1}), we obtain
\begin{eqnarray}\label{alphaeq1}
a_{n}x_{n}=\sum_{k=1}^{n}\frac{a_{n}}{n}y_{k}=(Cy)_{n} \quad\quad (n\in \mathbb{N}).
\end{eqnarray}
It follows from (\ref{alphaeq1}) that $ax=(a_{n}x_{n})\in \ell_{1}$ whenever $x=(x_{k})\in \int bv$ if and only if $Cy\in \ell_{1}$ whenever $y\in \ell_{1}$. By using Lemma \ref{duallem3} (iii), we obtain that $\left[\int bv\right]^{\alpha}=d(\ell_{1})$.
\end{proof}

\begin{thm}
$[d(bv)]^{\alpha}=\int \ell_{1}$
\end{thm}
\begin{proof}
As similar way in proof of Theorem \ref{alphathm}, if we take the matrix $\Sigma$ as defined by (\ref{matr2}) and define the matrix $D=(d_{nk})$ with $a_{n}x_{n}=\sum_{k=2}^{n}n.a_{n}.y_{k}=(Dy)_{n} $ for all $n\in \mathbb{N}$, using by the relation (\ref{transform2}), this gives us that $[d(bv)]^{\alpha}=\int \ell_{1}$.
\end{proof}
\begin{thm}\label{betathm}
$\left[\int bv\right]^{\beta}=d(bs)$
\end{thm}

\begin{proof}
Consider the equation
\begin{eqnarray}\label{beta1}
\sum_{k=1}^{n}a_{k}x_{k}=\sum_{k=1}^{n}a_{k}\left(k^{-1}\sum_{j=1}^{k}y_{j}\right)=\sum_{k=1}^{n}\left(\sum_{j=k}^{n}\frac{a_{j}}{j}\right)y_{k}=(Ey)_{n}
\end{eqnarray}
where $E=(e_{nk})$ is defined by
\begin{eqnarray}\label{beta2}
e_{nk}= \left\{ \begin{array}{ccl}
\sum_{j=k}^{n}j^{-1}a_{j}&, & \quad (0\leq k \leq n)\\
0&, & \quad (k>n)
\end{array} \right.
\end{eqnarray}
for all $n,k\in \mathbb{N}$. Then we deduce from Lemma \ref{duallem3} (ii) with (\ref{beta1}) that $ax=(a_{k}x_{k})\in cs$ whenever $x=(x_{k})\in \int bv$ if and only if $Ey\in c$ whenever $y=(y_{k})\in \ell_{1}$. Thus, $(a_{k})\in cs$ and $(a_{k})\in d(bs)$ by (\ref{deq1}) and (\ref{deq2}), respectively. Since the inclusion $d(bs)\subset cs$ holds, then, we have $(a_{k})\in d(bs)$, whence $\left[\int bv\right]^{\beta}=d(bs)$.
\end{proof}

\begin{lem}\cite{Goes}\label{betabv}
$(cs)^{\beta}= bv \Rightarrow [d(cs)]^{\beta}=\int bv$
\end{lem}

From  Theorem \ref{betathm} and Lemma \ref{betabv}, we have,

\begin{thm}
$(bv)^{\beta}=cs \Rightarrow [d(bv)]^{\beta}=\int cs$.
\end{thm}

\begin{thm}
$\left[\int bv\right]^{\gamma}=d(bs)$
\end{thm}
\begin{proof}
This can be obtained by analogy with the proof of Theorem \ref{betathm} with Lemma \ref{duallem3} (i) instead of Lemma \ref{duallem3} (ii). So we omit the details.
\end{proof}

\begin{thm}
$[d(bs)]^{\gamma}=\int bv$
\end{thm}

\section{Matrix Mappings on the spaces $\int bv$ and $d(bv)$}

In this section, we characterize some matrix transformations on the spaces $\int bv$ and $d(bv)$.\\

We shall write throughout for brevity that
\begin{eqnarray*}
&&\overline{a}_{nk}=k^{-1}\sum_{j=k}^{\infty}a_{nj}, \quad\quad \widetilde{a}_{nk}=k\sum_{j=k}^{\infty}a_{nj},\\
&&\widehat{a}_{nk}=n.a_{nk}-(n-1).a_{n-1,k}, \quad\quad \overrightarrow{a}_{nk}=n^{-1}.a_{nk}-(n-1)^{-1}.a_{n-1,k}
\end{eqnarray*}
for all $k,n \in \mathbb{N}$.

\begin{lem}\label{lemma1}\cite{AB}
Let $X, Y$ be any two sequence spaces, $A$ be an infinite matrix and $U$ a triangle matrix matrix.Then, $A\in (X: Y_{U})$ if and only if $UA\in (X:Y)$.
\end{lem}

\begin{thm}
Suppose that the entries of the infinite matrices $A=(a_{nk})$ and $F=(f_{nk})$ are connected with the relation
\begin{eqnarray}\label{mtrtrf1}
f_{nk}=\overline{a}_{nk}
\end{eqnarray}
for all $k,n\in \mathbb{N}$ and $Y$ be any given sequence space. Then, $A\in (\int bv:Y)$ if and only if $\{a_{nk}\}_{k\in \mathbb{N}}\in [\int bv]^{\beta}$ for all $n\in \mathbb{N}$ and $F\in (\ell_{1}:Y)$.
\end{thm}
\begin{proof}

Let $Y$ be any given sequence space. Suppose that (\ref{mtrtrf1}) holds between $A=(a_{nk})$ and $F=(f_{nk})$, and take into account that the spaces $\int bv$ and $\ell_{1}$ are norm isomorphic.\\

Let $A\in (\int bv:Y)$ and take any $y=(y_{k})\in \ell_{1}$. Then $\Gamma F$ exists and $\{a_{nk}\}_{k\in\mathbb{N}}\in \{\int bv\}^{\beta}$ which yields that (\ref{mtrtrf1}) is necessary and $\{f_{nk}\}_{k\in\mathbb{N}}\in (\ell_{1})^{\beta}$ for each $n\in \mathbb{N}$. Hence, $Fy$ exists for each $y\in \ell_{1}$ and thus \begin{eqnarray*}
\sum_{k}f_{nk}y_{k}=\sum_{k}a_{nk}x_{k}~ \textrm{ for all }~ n\in\mathbb{N},
\end{eqnarray*}
we obtain that $Fy=Ax$ which leads us to the consequence $F\in (\ell_{1}:Y)$.\\

Conversely, let $\{a_{nk}\}_{k\in\mathbb{N}}\in \{\int bv\}^{\beta}$ for each $n\in \mathbb{N}$ and $F\in (\ell_{1}:Y)$ hold, and take any $x=(x_{k})\in \int bv$. Then, $Ax$ exists. Therefore, we obtain from the equality
\begin{eqnarray*}
\sum_{k=1}^{m}a_{nk}x_{k}=\sum_{k=1}^{m}\left[k^{-1}\sum_{j=k}^{m}a_{nj}\right]y_{k}   ~ \textrm{ for all }~m,n\in \mathbb{N}
\end{eqnarray*}
as $m\rightarrow\infty$ that $Ax=Fy$ and this shows that $F\in(\ell_{1}:Y)$. This completes the proof.

\end{proof}

\begin{thm}
Suppose that the entries of the infinite matrices $A=(a_{nk})$ and $G=(g_{nk})$ are connected with the relation
\begin{eqnarray}\label{mtrtrf2}
g_{nk}=\widetilde{a}_{nk}
\end{eqnarray}
for all $k,n\in \mathbb{N}$ and $Y$ be any given sequence space. Then, $A\in (d(bv):Y)$ if and only if $\{a_{nk}\}_{k\in \mathbb{N}}\in [d(bv)]^{\beta}$ for all $n\in \mathbb{N}$ and $G\in (\ell_{1}:Y)$.
\end{thm}

\begin{thm}
Suppose that the entries of the infinite matrices $A=(a_{nk})$ and $H=(h_{nk})$ are connected with the relation
\begin{eqnarray}\label{mtrtrf3}
h_{nk}=\widehat{a}_{nk}
\end{eqnarray}
for all $k,n\in \mathbb{N}$ and $Y$ be any given sequence space. Then, $A\in (Y:\int bv)$ if and only if $M\in (Y:\ell_{1})$.
\end{thm}

\begin{proof}
Let $z=(z_{k})\in Y$ and consider the following equality
\begin{eqnarray*}
\sum_{k=0}^{m}\widehat{a}_{nk}z_{k}=\sum_{k=0}^{m}(n.a_{nk}-(n-1)a_{n-1,k})z_{k} \quad ~\textrm{ for all, }~   m,n\in \mathbb{N}
\end{eqnarray*}
which yields that as $m\rightarrow \infty$ that  $(Hz)_{n}=\{\Gamma(Az)\}_{n}$ for all $n\in \mathbb{N}$. Therefore, one can observe from here that $Az\in \int bv$ whenever $z\in Y$ if and only if $Hz\in \ell_{1}$ whenever $z\in Y$.
\end{proof}

\begin{thm}
Suppose that the entries of the infinite matrices $A=(a_{nk})$ and $M=(m_{nk})$ are connected with the relation
\begin{eqnarray}\label{mtrtrf4}
m_{nk}=\overrightarrow{a}_{nk}
\end{eqnarray}
for all $k,n\in \mathbb{N}$ and $Y$ be any given sequence space. Then, $A\in (Y:d(bv))$ if and only if $F\in (Y:\ell_{1})$.
\end{thm}

\begin{lem}\label{lemmtr1}
\begin{itemize}
\item[(i)] $A\in (\ell_{1}:bs)$ if and only if
\begin{eqnarray} \label{deq4}
\sup_{k,m \in \mathbb{N}}\left|\sum_{n=0}^{m}a_{nk}\right| < \infty.
\end{eqnarray}

\item[(ii)] $A\in (\ell_{1}:cs)$ if and only if (\ref{deq4}) holds, and
\begin{eqnarray}\label{deq5}
\sum_{n}a_{nk} ~ \textrm{ convergent for each }~k\in\mathbb{N}.
\end{eqnarray}

\item[(iii)] $A\in (\ell_{1}:c_{0}s)$ if and only if (\ref{deq4}) holds, and
\begin{eqnarray}\label{deq6}
\sum_{n}a_{nk}=0 ~ \textrm{ for each }~ k\in\mathbb{N}.
\end{eqnarray}
\end{itemize}
\end{lem}

\begin{lem}\label{lemmtr2}
\begin{itemize}

\item[(i)] $A\in (\ell_{\infty}:\ell_{1})=(c:\ell_{1})=(c_{0}:\ell_{1})$ if and only if
\begin{eqnarray} \label{deq0}
\sup_{N,K\in \mathcal{F}}\left|\sum_{n\in N}\sum_{k\in K}(a_{nk}-a_{n,k+1})\right|<\infty
\end{eqnarray}

\item[(ii)] $A\in (bs:\ell_{1})$ if and only if
\begin{eqnarray} \label{deq7}
\lim_{k}a_{nk}=0 ~ \textrm{ for each }~ n\in\mathbb{N}.
\end{eqnarray}
\begin{eqnarray} \label{deq8}
\sup_{N,K\in \mathcal{F}}\left|\sum_{n\in N}\sum_{k\in K}(a_{nk}-a_{n,k+1})\right|<\infty
\end{eqnarray}

\item[(iii)] $A\in (cs:\ell_{1})$ if and only if
\begin{eqnarray} \label{deq9}
\sup_{N,K\in \mathcal{F}}\left|\sum_{n\in N}\sum_{k\in K}(a_{nk}-a_{n,k-1})\right|<\infty
\end{eqnarray}

\item[(iv)] $A\in (c_{0}s:\ell_{1})$ if and only if (\ref{deq8}) holds.
\end{itemize}
\end{lem}

Now, we can give the following results:
\begin{cor}
The following statements hold:
\begin{itemize}
 \item[(i)] $A=(a_{nk})\in (\int bv:\ell_{\infty})$ if and only if $\{a_{nk}\}_{k\in \mathbb{N}} \in \{\int bv\}^{\beta}$ for all $n \in \mathbb{N}$ and (\ref{deq1}) holds with $\overline{a}_{nk}$ instead of ${a}_{nk}$.
    \item[(ii)] $A=(a_{nk})\in (\int bv:c)$ if and only if $\{a_{nk}\}_{k\in \mathbb{N}} \in \{\int bv\}^{\beta}$ for all $n \in \mathbb{N}$ and (\ref{deq1}) and  (\ref{deq2}) hold with $\overline{a}_{nk}$ instead of ${a}_{nk}$.
  \item[(iii)] $A\in (\int bv:c_{0})$ if and only if $\{a_{nk}\}_{k\in \mathbb{N}}\in \{\int bv\}^\beta$ for all $n\in \mathbb{N}$
  and (\ref{deq1}) and  (\ref{deq2}) hold with $\alpha_{k}=0$ as $\overline{a}_{nk}$ instead of $a_{nk}$.
      \item[(iv)] $A=(a_{nk})\in (\int bv:bs)$ if and only if $\{a_{nk}\}_{k\in \mathbb{N}} \in \{\int bv\}^{\beta}$ for all $n \in \mathbb{N}$ and (\ref{deq4}) holds with $\overline{a}_{nk}$ instead of ${a}_{nk}$.
  \item[(v)] $A=(a_{nk})\in (\int bv:cs)$ if and only if $\{a_{nk}\}_{k\in \mathbb{N}} \in \{\int bv\}^{\beta}$ for all $n \in \mathbb{N}$ and (\ref{deq4}), (\ref{deq5}) hold with $\overline{a}_{nk}$ instead of ${a}_{nk}$.
  \item[(vi)] $A=(a_{nk})\in (\int bv:c_{0}s)$ if and only if $\{a_{nk}\}_{k\in \mathbb{N}} \in \{\int bv\}^{\beta}$ for all $n \in \mathbb{N}$ and (\ref{deq4}), (\ref{deq6}) hold with $\overline{a}_{nk}$ instead of ${a}_{nk}$.

\end{itemize}
\end{cor}

\begin{cor}
The following statements hold:
\begin{itemize}
 \item[(i)] $A=(a_{nk})\in (d(bv):\ell_{\infty})$ if and only if $\{a_{nk}\}_{k\in \mathbb{N}} \in \{d(bv)\}^{\beta}$ for all $n \in \mathbb{N}$ and (\ref{deq1}) holds with $\widetilde{a}_{nk}$ instead of ${a}_{nk}$.
    \item[(ii)] $A=(a_{nk})\in (d(bv):c)$ if and only if $\{a_{nk}\}_{k\in \mathbb{N}} \in \{d(bv)\}^{\beta}$ for all $n \in \mathbb{N}$ and (\ref{deq1}) and  (\ref{deq2}) hold with $\widetilde{a}_{nk}$ instead of ${a}_{nk}$.
  \item[(iii)] $A\in (d(bv):c_{0})$ if and only if $\{a_{nk}\}_{k\in \mathbb{N}}\in \{d(bv)\}^\beta$ for all $n\in \mathbb{N}$
  and (\ref{deq1}) and  (\ref{deq2}) hold with $\alpha_{k}=0$ as $\widetilde{a}_{nk}$ instead of $a_{nk}$.
      \item[(iv)] $A=(a_{nk})\in (d(bv):bs)$ if and only if $\{a_{nk}\}_{k\in \mathbb{N}} \in \{d(bv)\}^{\beta}$ for all $n \in \mathbb{N}$ and (\ref{deq4}) holds with $\widetilde{a}_{nk}$ instead of ${a}_{nk}$.
  \item[(v)] $A=(a_{nk})\in (d(bv):cs)$ if and only if $\{a_{nk}\}_{k\in \mathbb{N}} \in \{d(bv)\}^{\beta}$ for all $n \in \mathbb{N}$ and (\ref{deq4}), (\ref{deq5}) hold with $\widetilde{a}_{nk}$ instead of ${a}_{nk}$.
  \item[(vi)] $A=(a_{nk})\in (d(bv):c_{0}s)$ if and only if $\{a_{nk}\}_{k\in \mathbb{N}} \in \{\int bv\}^{\beta}$ for all $n \in \mathbb{N}$ and (\ref{deq4}), (\ref{deq6}) hold with $\widetilde{a}_{nk}$ instead of ${a}_{nk}$.
\end{itemize}
\end{cor}

\begin{cor}
We have:
\begin{itemize}
  \item[(i)] $A=(a_{nk})\in (\ell_{\infty}:\int bv)=(c:\int bv)=(c_{0}:\int bv)$ if and only if (\ref{deq0}) hold with $\widehat{a}_{nk}$ instead of ${a}_{nk}$.
  \item[(ii)] $A=(a_{nk})\in (bs:\int bv)$ if and only if (\ref{deq7}) and (\ref{deq8}) hold with $\widehat{a}_{nk}$ instead of ${a}_{nk}$.
  \item[(iii)] $A=(a_{nk})\in (cs:\int bv)$ if and only if (\ref{deq9}) holds with $\widehat{a}_{nk}$ instead of ${a}_{nk}$.
  \item[(iv)] $A=(a_{nk})\in (c_{0}s:\int bv)$ if and only if (\ref{deq8}) holds with $\widehat{a}_{nk}$ instead of ${a}_{nk}$.

\end{itemize}
\end{cor}

\begin{cor}
We have:
\begin{itemize}
  \item[(i)] $A=(a_{nk})\in (\ell_{\infty}:d(bv))=(c:d(bv))=(c_{0}:d(bv))$ if and only if (\ref{deq0}) hold with $\overrightarrow{a}_{nk}$ instead of ${a}_{nk}$.
  \item[(ii)] $A=(a_{nk})\in (bs:d(bv))$ if and only if (\ref{deq7}) and (\ref{deq8})  hold with $\overrightarrow{a}_{nk}$ instead of ${a}_{nk}$.
  \item[(iii)] $A=(a_{nk})\in (cs:d(bv))$ if and only if (\ref{deq9}) holds with $\overrightarrow{a}_{nk}$ instead of ${a}_{nk}$.
    \item[(iv)] $A=(a_{nk})\in (c_{0}s:d(bv))$ if and only if (\ref{deq8}) holds with $\widehat{a}_{nk}$ instead of ${a}_{nk}$.
\end{itemize}
\end{cor}

\section{Conclusion}
Goes and Goes \cite{Goes} introduced the integrated and differentiated sequence spaces. Subramanian et.al. \cite{subraoguru} gave the integrated rate space $\int \ell_{\pi}$ and studied some properties of this space. And  they also characterized the matrix classes $\left(\int\ell_{\pi}:Y\right)$, where $Y=\{\ell_{\infty}, c,c_{0}, \ell_{p}, bv, bv_{0}, bs, cs, \ell_{\rho}, \ell_{\pi}\}$. There are no studies on differentiated sequence spaces.\\

In this paper, we studied some properties of integrated and differentiated sequence spaces. We compute the alpha-, beta- and gamma-duals of these spaces. For $Y=\{\ell_{\infty}, c,c_{0}, bs, cs, c_{0}s\}$, we characterize matrix classes $(\int bv:Y), (d(bv):Y)$ and $(Y: \int bv), (Y: d(bv))$ in the last section.\\

We should note from now on that the investigation of the domain of some particular limitation matrices, namely
Cesàro means of order one, Euler means of order r, Riesz means, Nörlund means, the double band matrix $B(r,s)$,
the triple band matrix $B(r,st), $etc., in the spaces $\int bv$ and $d(bv)$ will lead us to new
results which are not comparable with the present results. If we can choose different sequence spaces for the space $Y$,
it can study new matrix characterizations of $(\int bv:Y), (d(bv):Y)$ and $(Y: \int bv), (Y: d(bv))$. Also the spaces $\int bv$ and $d(bv)$
can be defined by a index $p$ and paranormed sequence spaces as $p=(p_{k})$ is a sequence of strictly positive numbers.

\end{document}